\documentclass[12pt, letter]{article}

\usepackage{thkim1011}
\usepackage{tikz}
\usepackage{tikz-cd}
\usepackage{xcolor}
\usepackage{indentfirst}

\renewcommand{\C}{\mathcal{C}}
\renewcommand{\S}{\mathcal{S}}

\newcommand{\cht}{\text{coht}}
\newcommand{\monster}{Theorem 3.15 } 

\let\svwidehat\widehat
\renewcommand\widehat[1]{\ifx\relax#1\relax\svwidehat{\phantom{x}}\else\svwidehat{#1}\fi}

\renewcommand{\hat}{\widehat}

\title{Cardinalities of Prime Spectra of Precompletions}

\author{Erica Barrett, Emil Graf, S. Loepp, Kimball Strong, and Sharon Zhang}

\begin{document}
\maketitle

\begin{abstract}
    Given a complete local (Noetherian) ring $T$, we find necessary and sufficient conditions on $T$ such that there exists a local domain $A$ with $|A| < |T|$ and $\widehat{A} = T$, where $\widehat{A}$ denotes the completion of $A$ with respect to its maximal ideal. We then find necessary and sufficient conditions on $T$ such that there exists a domain $A$ with $\widehat{A} = T$ and $|\Spec(A)| < |\Spec(T)|$. Finally, we use ``partial completions'' to create local rings $A$ with $\hat{A} = T$ such that $\Spec(A)$ has varying cardinality in different varieties. 
\end{abstract}

\section{Introduction}
\par Since the structure of complete local rings is better understood than that of local rings that are not complete, it is important to study the relationship between a local ring and its completion. To examine this relationship, we work backwards. 
Given a complete local ring $T$, we wish to determine when there exists a local ring $A$ with certain ``prescribed properties" such that the completion of $A$ is $T$. We call this local ring $A$ a ``precompletion" of $T$. Importantly, a complete local ring $T$ usually has multiple distinct precompletions, each of which may satisfy different properties.
\par Past results have focused on precompletions satisfying prescribed properties such as being integral domains (see \cite{lech}), UFDs (see \cite{heitmann93}), and/or rings having prescribed generic formal fibers (see, for example, \cite{pippa}). We focus on properties concerning the prime spectrum (or simply spectrum) of $A$, denoted by $\Spec(A)$. More specifically, given a complete local ring $T$, we focus on the possible cardinalities of its precompletions as well as the possible cardinalities of the spectra of its precompletions.
\par First, we consider the case where the precompletion is an integral domain. In \cite{lech}, Lech completely characterizes complete local rings that are the completions of local domains; specifically, a complete local ring $T$ has a local domain precompletion if and only if
\begin{enumerate}
    \item no integer of $T$ is a zerodivisor, and
    \item the maximal ideal of $T$ is equal to (0) or is not an associated prime ideal of $T$.
\end{enumerate}
We build upon this result and determine when $T$ is the completion of a local domain $A$ with cardinality smaller than $|T|$. 
We find that the existence of such a domain $A$ requires only the additional condition that $|T/M| < |T|$, where $M$ is the maximal ideal of $T$.
We then assume that the dimension of $T$ is at least two and answer the analogous question for the cardinality of $\Spec(A)$. That is, when does $T$ have a precompletion $A$ such that $|\Spec(A)| < |T|$? We show that, if the dimension of $T$ is at least two, then $|\Spec(T)| = |T|$ and so this question is equivalent to asking under what conditions does $T$ have a precompletion $A$ such that $|\Spec(A)| < |\Spec(T)|$?  Again, we find the existence of such an $A$ requires only the additional condition $|T/M| < |T|$.
\par Next, we examine the effect of ``partial completions" (completions at a non-maximal ideal) on spectra of rings. 
Previous work in \cite{wiegand} shows that such rings need not have uniform cardinalities. 
That is, different varieties of the prime spectrum can have different cardinalities; natural examples arise in the case of mixed polynomial-power series rings. 
In general, this approach allows us to obtain precompletions with different cardinalities in different varieties, and we apply this technique to find precompletions with ``unbalanced cardinalities.''  As an illustration, see Example \ref{standard}.
\par All rings in this paper are commutative with identity element $1$.  We use ``local" to describe a Noetherian ring with exactly one maximal ideal and ``quasi-local" to describe a ring (not necessarily Noetherian) with exactly one maximal ideal. When we say $(T,M)$ is a local ring, we mean that $T$ is a local ring with maximal ideal $M$.  Finally, unless otherwise noted, if $(A,{\bf m})$ is a local ring, we use $\widehat{A}$ to denote the completion of $A$ at ${\bf m}$.

\section{Cardinalities of Domain Precompletions}

\par We begin by investigating the cardinalities of integral domain precompletions and their spectra. Domains are relatively simple compared to general rings because they have only one minimal prime ideal, namely, (0). This simplicity makes them desirable to study.

\par The following theorem is a result from \cite{lech}, and it gives necessary and sufficient conditions for a complete local ring to be the completion of a local domain. \\

\begin{thm} [\cite{lech} Theorem 1]\label{Lech}
A complete local ring $(T, M)$ is the completion of a local integral domain if and only if 
\begin{enumerate}[label={(\roman*)}]
    \item No integer of T is a zerodivisor
    \item Unless equal to $(0)$, $M \notin \Ass(T)$.
\end{enumerate}
\end{thm}

In this section, we determine when a given complete local ring $T$ has a domain precompletion $A$ with $|A| < |T|$. As a corollary, we find necessary and sufficient conditions for $T$ to be the completion of a countable local domain.
To do this, we first introduce PB-subrings. A PB-subring of $T$ is a subring that satisfies cardinality and intersection properties that we will need to construct our desired ``small" domain precompletions.

\begin{dfn}\label{pb-subring}
    Let $(T, M)$ be a complete local ring and $(R, R \cap M)$ a quasi-local subring of $T$. Suppose the following conditions hold:
    \begin{enumerate}[label={(\roman*)}]
       \item $|R| \leq \sup\{\aleph_0, |T/M|\}$, with equality implying $T/M$ is countable
       \item $R \cap Q = (0)$ for every $Q \in \Ass(T)$
    \end{enumerate}
    Then we call $R$ a \textbf{partially-built-subring} of $T$, or \textbf{PB-subring} for short. 
\end{dfn}

The following proposition is from \cite{heitmann94}. \\

\begin{prop}[\cite{heitmann94} Proposition 1] \label{cpm}
    If $(A, A \cap M)$ is a quasi-local subring of a complete local ring $(T, M)$, the map $A \rightarrow T/M^2$ is onto, and $IT \cap A = I$ for all finitely generated ideals $I$ of $A$, then $A$ is Noetherian and the natural homomorphism $\widehat{A} \rightarrow T$ is an isomorphism.
\end{prop}

By constructing our rings so that the two conditions in the above proposition are satisfied, we ensure that our rings have the desired completion.
\par The next two lemmas are from \cite{heitmann93}. They will allow us to find elements transcendental over a subset of elements in our rings, which we then adjoin to build a larger ring. \\

\begin{lem} [\cite{heitmann93} Lemma 2]\label{heitmann93.2}
    Let $T$ be a complete local ring with maximal ideal $M$, $C$ a countable set of primes in $\Spec(T)$ such that $M \notin C$, and $D$ a countable set of elements in $T$. If $I$ is an ideal in $T$ which is contained in no single $P \in C$, then $I \nsubseteq \bigcup \{ r + P \mid P \in C, r \in D \}$.
\end{lem}

\begin{lem}[\cite{heitmann93} Lemma 3]\label{heitmann93.3}
    Let $(T, M)$ be a local ring. Let $C \subseteq \Spec(T)$, let $I$ be an ideal such that $I \nsubseteq P$ for every $P \in C$, and let $D$ be a subset of $T$. Suppose $|C \times D| < |T/M|$. Then $I \nsubseteq \bigcup \{ r + P \mid P \in C, r \in D \}$.
\end{lem}

Lemmas \ref{Deb 3.6}, \ref{Deb 3.8}, \ref{Deb 3.9}, and \ref{Deb 3.11} and their proofs are quoted almost directly from \cite{deb} with some minor modifications. Because \cite{deb} is not publicly available, we include all details of the proofs. These lemmas provide the framework for Theorem \ref{general domain construction}, which is our own original result.\\

\begin{lem}\label{Deb 3.6}
    Let $(T, M)$ be a complete local ring with $M \notin \Ass(T)$. Let $R$ be a PB-subring of $T$ and let $C = \Ass(T)$. For each $P \in C$, define $D_{(P)}$ to be a full set of coset representatives of $T/P$ that are algebraic over $R$. Let $D = \bigcup_{P \in C} D_{(P)}$. If $J$ is an ideal of $T$ which is not contained in any single $P \in C$, then $J \nsubseteq \bigcup \{ r + P \mid P \in C, r \in D \}$.
\end{lem}

\begin{proof}
    Since $T$ is Noetherian, $C$ must be finite. First, suppose $|R| \leq \aleph_0$. Since $R$ is countable, given $P$, $D_{(P)}$ must also be countable. Hence, $|D| \leq \aleph_0$, as the finite union of countable sets is countable. By hypothesis, $M \notin \Ass(T)$. So by Lemma \ref{heitmann93.2}, since $J$ is an ideal of $T$ which is contained in no single $P$ in $C$, then $J \nsubseteq \bigcup \{ r + P \mid P \in C, r \in D \}$.
    \par Now, assume $|R| > \aleph_0$. Then, by property (i) of PB-subrings, $|T/M| > |R|$. Since $R$ is infinite, the cardinality of the set of polynomials with coefficients in $R$ is equal to the cardinality of $R$. Thus, $|D_{(P)}|\leq |R|$ for $P \in \Spec(T)$, and, since $C$ is finite, $|D| \leq |R|$. Thus, $|C \times D| \leq |R| < |T/M|$. So, by Lemma \ref{heitmann93.3}, if $J$ is an ideal of $T$ which is contained in no single $P$ in $C$, then $J \nsubseteq \bigcup \{ r + P \mid P \in C, r \in D \}$.
\end{proof}

Proposition \ref{cpm} requires that for every finitely generated ideal $I$ of $A$, $IT \cap A = I$, a condition sometimes referred to as ``closing up ideals." 
The following lemma will allow us to ``close up" any ideal $I$ in a $PB$-subring. The statement of this lemma and its proof are very similar to those of Lemma 2.6 in \cite{pippa}. \\

\begin{lem}\label{Deb 3.5}
    Let $(T, M)$ be a complete local ring with $M \notin \Ass(T)$ and $R$ a PB-subring of $T$. Suppose $I$ is a finitely generated ideal of $R$ and $c \in IT \cap R$. Then there exists a PB-subring S of $T$ with $R \subset S \subset T$ and $c \in IS$.  Furthermore, $|S| \leq \sup\{|R|,\aleph_0\}$.
\end{lem}

\begin{proof}
    The proof is identical to that of Lemma 2.6 in \cite{pippa}; in addition to the conclusions in that paper, we require that our cardinality condition, property (i) of PB-subrings, is satisfied. Since $S$ is obtained by adjoining finitely many elements to $R$, this condition, as well as the condition $|S| \leq \sup\{|R|,\aleph_0\}$ follows immediately.  
\end{proof}

\begin{lem}\label{Deb 3.8}
    Let $(T, M)$ be a complete local ring with $M \notin \Ass(T)$. Let $R$ be a PB-subring of $T$ and $u \in T$. Then there exists a PB-subring $S$ of $T$ such that $R \subset S \subset T$, there is some $c \in S$ with $u - c \in M^2$, and $|S| = \sup\{|R|,\aleph_0\}$.
\end{lem}

\begin{proof}
    Let $C = \Ass(T)$. 
    Suppose $P \in C$. Define $D_{(P)}$ to be a full set of coset representatives $t + P$ that make $(u + t) + P$ algebraic over $R$ as an element of $T/P$. 
    Note that our hypotheses give us $M \nsubseteq P$ for every $P \in C$. 
    So, $M^2 \nsubseteq P$ for every $P \in C$. Let $D = \bigcup_{P \in C} D_{(P)}$. 
    Using Lemma \ref{Deb 3.6} with $J = M^2$, we may find an $x \in M^2$ such that $x \notin \bigcup \{ r + P \mid r \in D, P \in C \}$. 
    Consider $S = R[x + u]_{(R[x + u] \cap M)}$. 
    It is clear that $R \subset S$. 
    Moreover, all elements of $R[x + u]$ which are not contained in $M$ are units of $T$. 
    Thus, $S \subset T$.  Note that if $c = x + u$ then $u - c \in M^2$.  In addition, by definition, $|S|=\sup\{|R|,\aleph_0\}$, and so $S$ satisfies property (i) of PB-subrings.
  
    \par Let $a \in R[x + u] \cap P$ for some $P \in \Ass(T)$. 
    Since $a \in R[x + u]$, we may write
    \[
    a = r_0 + r_1(x + u) + r_2(x + u)^2 + \cdots + r_n(x + u)^n
    \]
    for $r_0, r_1, \ldots, r_n \in R$. 
    As $x \notin \bigcup \{ r + P \mid r \in D, P \in C \}$, $x + u + P$ is transcendental over $R$ and so $r_i \in P$ for all $i = 1, 2, \ldots, n$. 
    However, since $R$ is a PB-subring of $T$, $R \cap P = (0)$ so $R[x + u] \cap P = (0)$. 
    Thus $S\cap P = (0)$ and it follows that $S$ satisfies property (ii) of the definition of PB-subrings.
\end{proof}

\begin{lem}\label{Deb 3.9}
    Let $(T, M)$ be a complete local ring with $M \notin \Ass(T)$,  and let $\Omega$ be a well-ordered set such that $|\Omega| \leq \sup\{\aleph_0, |T/M|\}.$
    Suppose $\{ R_{\alpha} \mid \alpha \in \Omega \}$ is an ascending collection of PB-subrings. Then $R = \bigcup_{\alpha \in \Omega} R_{\alpha}$ is a PB-subring except that it may be the case that $|R|=|T/M|$ when $T/M$ is uncountable. Furthermore, if there exists $\lambda$ such that $|R_{\alpha}| \leq \lambda$ for all $\alpha \in \Omega$, then $|R| \leq |\Omega|\lambda$ and if $T/M$ is countable or if $|\Omega|<|T/M|$ and $\lambda < |T/M|$ then $R$ is a PB-subring of $T$.
\end{lem}

\begin{proof}
    Note that $R$ is a quasi-local subring of $T$ with maximal ideal $R \cap M$.
    \par Now let $r \in R \cap Q$ where $Q \in \Ass(T)$. Then $r \in R$ which implies that $r \in R_{\alpha}$ for some $\alpha \in \Omega$. Since $R_{\alpha}$ is a PB-subring, we have $R_{\alpha} \cap Q = (0)$. Hence $r = 0$, so $R \cap Q = (0)$ and $R$ satisfies property (ii) of PB-subrings.
    \par Now we prove that $|R| \leq \sup\{\aleph_0, |T/M|\}$. Since $R = \bigcup_{\alpha \in \Omega} R_{\alpha}$,
    \[
    |R| \leq \Sigma_{\alpha \in \Omega}|R_{\alpha}| \leq |\Omega| \sup_{\alpha \in \Omega}|R_{\alpha}| \leq [\sup\{\aleph_0, |T/M|\}]^2 = \sup\{\aleph_0, |T/M|\}.
    \]
    \par Note that if $T/M$ is countable, then the inequality above implies that $R$ is a PB-subring. Now we assume that there exists $\lambda$ such that $|R_{\alpha}| \leq \lambda$ for all $\alpha \in \Omega$. Then:
    \[
    |R| \leq \Sigma_{\alpha \in \Omega}|R_{\alpha}| \leq |\Omega| \lambda
    \]
    Thus if $|\Omega|<|T/M|$ and $|\lambda|<|T/M|$, then:
    \[
    |R| \leq |\Omega| \lambda < [\sup\{\aleph_0, |T/M|\}]^2 = \sup\{\aleph_0, |T/M|\}.
    \]
    so property (i) of PB-subrings is satisfied, and $R$ is a PB-subring of $T$.
    \end{proof}
    
\begin{dfn}\label{Deb 3.10}
    Let $\Omega$ be a well-ordered set and let $\alpha \in \Omega$. We denote 
    $$\gamma (\alpha) = \sup\{ \beta \in \Omega \mid \beta < \alpha \}.$$
\end{dfn}

\begin{lem}\label{Deb 3.11}
    Let $(T, M)$ be a complete local ring with $M \notin \Ass(T)$, and let $\bar{t} \in T/M^2$. Let $R$ be a PB-subring of $T$. Then there exists a PB-subring $S$ of $T$ such that the following conditions hold:
    \begin{enumerate}[label={(\roman*)}]
      \item $R \subset S \subset T$ with $|S| = \sup\{|R|,\aleph_0\}$
       \item $\bar{t} \in \textup{Im}(S \rightarrow T/M^2)$
       \item For every finitely generated ideal I of S, we have $IT \cap S = I$.
    \end{enumerate}
\end{lem}

\begin{proof}
    First, use Lemma \ref{Deb 3.8} to find a PB-subring $R_0$ such that $R \subset R_0 \subset T$, $\bar{t} \in \text{Im}(R_0 \rightarrow T/M^2)$, and $|R_0|=\sup\{|R|,\aleph_0\}$. We will construct $S$ to contain $R_0$, so $R \subset S \subset T$ and condition (ii) will follow automatically. Let
    $$\Omega = \{ (I,c) \mid I \text{ a finitely generated ideal of } R_0 \text{ and } c \in IT \cap R_0\}.$$
    If $c \in R_0$ then $(R_0,c) \in \Omega$. Thus, $|R_0| \leq |\Omega|$. Since the cardinality of the set of finite subsets of $R_0$ is $|R_0|$, $|\Omega| \leq |R_0|$ and thus $|R_0| = |\Omega|$. Note that either $T/M$ is countable or $|R_0|<|T/M|$.
    \par Well-order $\Omega$ so that it does not have a maximal element and let 0 denote its initial element. Now, we recursively define a family of PB-subrings, beginning with $R_0$. If $\gamma (\alpha) \neq \alpha$ and $\gamma (\alpha) = (I, c)$, then choose $R_{\alpha}$ to be the PB-subring extension of $R_{\gamma (\alpha )}$ obtained from Lemma \ref{Deb 3.5}, so that $c \in IR_{\alpha}$ and $|R_{\alpha}| \leq \sup \{|R_{\gamma (\alpha)}|,\aleph_0\}$. 
    If $\gamma (\alpha ) = \alpha$, choose $R_{\alpha} = \bigcup_{\beta < \alpha} R_{\beta}$, noting that $R_{\alpha}$ is a PB-subring with $|R_{\alpha}| \leq |\Omega||R_0| = |R_0|$. Set $S_1 = \bigcup R_{\alpha}$. By Lemma \ref{Deb 3.9}, $S_1$ is a PB-subring of $T$ and $|S_1| \leq |\Omega||R_0| = |R_0|$.  Hence, $|S_1| = |R_0|$.
   By our construction, if $I$ is a finitely generated ideal of $R_0$ and $c \in IT \cap R_0$, then, for some $\alpha \in \Omega$, we have $c \in IR_{\alpha} \subseteq IS_1$ and hence $IT \cap R_0 \subseteq IS_1$. 
    \par We repeat this process to obtain a PB-subring extension $S_2$ of $S_1$ such that $IT \cap S_1 \subseteq IS_2$ for every finitely generated ideal $I$ of $S_1$ and $|S_2| = |R_0|$, and continuing further results in an ascending chain $R_0 \subset S_1 \subset \ldots $ such that $|S_n| = |R_0|$ for every $n$, and $IT \cap S_n \subseteq IS_{n + 1}$ for every finitely generated ideal $I$ of $R_n$. 
    Then by Lemma \ref{Deb 3.9}, $S = \bigcup S_i$ satisfies $|S| = |R_0| = \sup\{|R|,\aleph_0\}$ and $S$ is a PB-subring because either $T/M$ is countable or $|R_0| < |T/M|$.
    If $I$ is a finitely generated ideal of $S$, then some $S_n$ contains a generating set for $I$, say $y_1, \ldots, y_k$. If $c \in IT \cap S$, then $c \in S_m$ for some $m \geq n$. So $c \in  (y_1, \ldots, y_k)T \cap S_m$, and therefore $c \in (y_1, \ldots, y_k)S_{m + 1} \subseteq I$. Thus $IT \cap S = I$, so condition (iii) holds.
\end{proof}

The following lemma is a well-known result, but we include the proof for completeness. \\

\begin{lem} \label{residue field powers}
    Let $(A,M)$ be a local ring. If $A/M$ is finite, then $A/M^n$ is finite for all $n$. If $A/M$ is infinite, then $|A/M^n| = |A/M|$ for all $n$.
\end{lem}

\begin{proof}
    We have that
    \[
    |A/M^n| \le |A/M|\cdot|M/M^2| \cdots |M^{n-1}/M^n| = |A/M|\cdot |A/M|^{a_2} \cdots |A/M|^{a_n}
    \]
    Where $a_i$ is the dimension of $M^{i-1}/M^i$ as a vector space over $A/M$. Since $A$ is Noetherian, each $a_i$ is finite and so the result follows.
\end{proof}


We are now equipped to show our main result of this section. \\

\begin{thm} \label{general domain construction}
Let $(T, M)$ be a complete local ring such that
    \begin{enumerate}[label={(\roman*)}]
    \item No integer of $T$ is a zerodivisor
    \item Unless equal to (0), $M \notin \Ass(T)$.
    \end{enumerate}
If $T/M$ is infinite, then $T$ is the completion of a local domain $A$ such that $|A| = |T/M|$. If $T/M$ is finite, then $T$ is the completion of a countable domain.
\end{thm}

\begin{proof}
\par If $M = (0)$, then $|T| = |T/M|$ and $T$ is a field, so $T$ is a domain precompletion of itself  and the result follows immediately.
\par Now assume that $M \neq (0)$ so that, by condition (ii), we have $M \notin \Ass(T)$. Note that by Theorem \ref{residue field powers}, if $T/M$ is infinite, then $|T/M| = |T/M^2|$, and $T/M^2$ is finite otherwise. Let $\Omega = T/M^2$, and well-order $\Omega$ such that every element has strictly fewer than $|\Omega|$ predecessors. Define $\mathcal{B}$ to be an indexing set for $\Omega$, and denote $0$ as the least element of $\mathcal{B}$. 
We define a family of PB-subrings indexed by $\mathcal{B}$ as follows: let $R_0'$ be the prime subring of $T$ and $R_0$ be $R_0'$ localized at $R_0' \cap M$. Note that, by condition (i), $R_0$ is a PB-subring of $T$. For each $R_\lambda$, if $\gamma(\lambda) < \lambda$, use Lemma \ref{Deb 3.11} to construct $R_\lambda$ as a PB-subring extension of $R_{\gamma(\lambda)}$ such that 
    \begin{enumerate}[label={(\roman*)}]
      \item $R_{\gamma(\lambda)} \subset R_\lambda \subset T$ with $|R_{\lambda}| = \sup\{|R_{\gamma{(\lambda)}}|,\aleph_0\}$
       \item $\bar{t}_{\gamma(\lambda)} \in \Omega$ is in $\text{Im}(R_\lambda \rightarrow T/M^2)$
       \item For every finitely generated ideal $I$ of $R_\lambda$, we have $IT \cap R_\lambda = I.$
    \end{enumerate}
    If $\gamma(\lambda) = \lambda$, then let $R_\lambda = \bigcup_{\beta < \lambda} R_\beta$, and note that $R_\lambda$ is a PB-subring of $T$ by Lemma \ref{Deb 3.9}.
    Then 
    $$
    A = \bigcup_{\lambda \in \mathcal{B}} R_\lambda
    $$
    is the desired domain. Note that, by Lemma \ref{Deb 3.9}, $A$ is a PB-subring of $T$ except that we could have $|A| = |T/M|$ when $T/M$ is uncountable.  
    In particular, if $T/M$ is infinite, then $|A| \leq |T/M|$.
    Also note that each PB-subring contains no zerodivisors of $T$, so $A$ contains no zerodivisors of $T$.
    \par We now use Proposition \ref{cpm} to show $A$ is Noetherian and $\widehat{A} = T$. By construction, $A \rightarrow T/M^2$ is surjective. Next, let $I = (a_1, \ldots, a_n)A$ be a finitely generated ideal of $A$ and $c \in IT \cap A$. Then, for some $\lambda \in \mathcal{B}$ such that $\lambda$ has a predecessor, we have $\{ c, a_1, \ldots, a_n \} \subseteq R_\lambda$. 
    In particular, this yields $c \in (a_1,\ldots,a_n)T \cap R_{\lambda} = (a_1,\ldots,a_n)R_\lambda \subseteq I$. Hence, $IT \cap A = I$ for all finitely generated ideals $I$ of $A$. 
    It follows by Proposition \ref{cpm} that $A$ is Noetherian and $\widehat{A} = T$. Therefore, if $T/M$ is infinite, the isomorphism $A/(M \cap A) \cong T/M$ implies that $|T/M| = |A/(M \cap A)| \leq |A| \leq |T/M|$, so $|A|=|T/M|$.  If $T/M$ is finite, then $|A| \leq \sup\{|T/M|, \aleph_0\}$ implies that $A$ is countable.
\end{proof}

\begin{cor} \label{small domain thm}
    Let $(T,M)$ be a complete local ring. Then $T$ is the completion of a local domain $A$ such that $|A| < |T|$ if and only if the following three conditions hold.
    \begin{enumerate}[label={(\roman*)}]
    \item No integer of $T$  is a zerodivisor
    \item $M \not \in \textup{Ass} \, (T)$
    \item $|T/M| < |T|$
    \end{enumerate}
\end{cor}

\begin{proof}
    \par Suppose that $T$ is the completion of a local domain $A$ with $|A| < |T|$. 
    Since $A/(A \cap M) \cong T/M$, we have that $|T/M| \le |A| < |T|$ and so (iii) follows.  As a result, $M$ cannot be the zero ideal, and so by Theorem \ref{Lech}, (i) and (ii) must both be true. 
        \par Now suppose that conditions (i), (ii), and (iii) hold.  By (iii), $M$ cannot be the zero ideal.  If $T/M$ is infinite, the result follows from Theorem  \ref{general domain construction}.
         Since $M \not \in \textup{Ass} \, (T)$, the dimension of $T$ must be at least one, and so, by Lemma 2.2 in \cite{dundon}, $T$ is uncountable. 
    If $T/M$ is finite, $T$ is the completion of a countable domain $A$ by Theorem \ref{general domain construction}.  As $T$ is uncountable and $A$ is countable, we have $|A| < |T|$.
\end{proof}

\begin{cor} \label{countable domain thm}
    Let $(T,M)$ be a complete local ring. Then $T$ is the completion of a countable local domain if and only if the following three conditions hold.
    \begin{enumerate}[label={(\roman*)}]
        \item No integer of $T$ is a zerodivisor
        \item Unless equal to $(0)$, $M \not \in \textup{Ass} \, (T)$
        \item $T/M$ is countable.
    \end{enumerate}
\end{cor}

\begin{proof}
    \par Suppose $T$ is the completion of a countable local domain $A$.  Conditions (i) and (ii) follow from Theorem \ref{Lech}, and condition (iii) follows since $|T/M| = |A/(A \cap M)| \leq |A|$.  If conditions (i), (ii), and (iii) hold, then $T$ is the completion of a countable local domain by Theorem \ref{general domain construction}.
\end{proof}
The following example justifies the generality of the statement of Theorem \ref{general domain construction} and Corollary \ref{small domain thm}, showing that there exist cases outside of the countable/uncountable divide. 
\begin{ex}
    Let $k$ be a field with cardinality $\omega_\omega$, where $\omega$ here is used in the ordinary sense of the ordinal numbers ($\omega$ being the order type of the natural numbers, $\omega_1$ the smallest uncountable ordinal, etc.). Then, by K\"onig's Theorem we have that $|k|^{\aleph_0} > |k|$. 
    Let $T = k[[x,y,z]]/(x^2 - y^2)$, which has cardinality $|k|^{\aleph_0}$. By Theorem \ref{general domain construction}, there is a domain $R$ such that $|R| < |T|$ and $\widehat{R} = T$. 
\end{ex}

\section{Cardinalities of Spectra of Domain Precompletions}
In this section, we investigate the cardinality of the spectra of domain precompletions of a given complete local ring. In particular, we prove that every complete local ring with dimension at least two and satisfying the conditions of Theorem \ref{Lech} is the completion of a local domain with an uncountable spectrum. Then, we will use the construction in the previous section to find necessary and sufficient conditions for a complete local ring $T$ to be the completion of a local domain with a countable spectrum, and, more generally, with a spectrum smaller than $|\Spec(T)|$.
\par We begin with some preliminary results that will help us generalize the Prime Avoidance Theorem. The following theorem is a generalization of Corollary 2.6 from \cite{sharpvamos} and is used in the proof of Theorem \ref{countable spec}. \\

\begin{thm} \label{Sharp 2.6}
    Let $(R, M)$ be a local ring. Let $\mathfrak{a}$ be an ideal of $R$ and let $\mathcal{I}$ be an indexing set such that for each $i \in \mathcal{I}$, $\mathfrak{b}_i$ is an ideal of $R$, where $\mathfrak{a} \subseteq \bigcup_{i \in \mathcal{I}} \mathfrak{b}_i$. Suppose that $|R/M|$ is infinite and $|R/M| > |\mathcal{I}|$. 
    Then $\mathfrak{a} \subseteq \mathfrak{b}_i$ for some i.
\end{thm}   

\begin{proof}
    Let $\{u_{\lambda}\}$ for $\lambda \in R/M$ be a set of representatives of the cosets of $R/M$. Let $x_1, \ldots, x_k$ be a set of generators for $\mathfrak{a}$. For each $\lambda \in R/M$, define
    \[
    y_\lambda = x_1 + u_\lambda x_2 + \cdots + u_\lambda^{k-1}x_k
    \]
    Note that $y_\lambda \in \mathfrak{a} \subseteq \bigcup_{i \in \mathcal{I}} \mathfrak{b}_i$. 
    Hence for all $\lambda \in R/M$, there is some $i \in \mathcal{I}$ such that $y_\lambda \in \mathfrak{b}_i$. 
    Define the sets $B_i = \{ \lambda \mid y_\lambda \in \mathfrak{b}_i \}$. 
    Then there is some $B_j$ with infinitely many elements; hence there is some $B_j$ containing $k$ elements $y_{\lambda_1}, \ldots, y_{\lambda_k}$. Consider the matrix $P = (u_{\lambda_i}^{j-1})$; i.e. the matrix
    \[
    \begin{pmatrix}
    1 & u_{\lambda_1} &  \dots & u_{\lambda_1}^{k-1} \\
    1 & u_{\lambda_2} & \dots & u_{\lambda_2}^{k-1} \\
    \vdots & \vdots & \ddots & \vdots \\
    1 & u_{\lambda_k} & \dots & u_{\lambda_k}^{k-1}
    \end{pmatrix}
    \]
    Then because the difference $u_{\lambda_i} - u_{\lambda_j}$ is a unit for $i \neq j$, the matrix is invertible. 
    Note that $P(x_1, \ldots, x_k)^T = (y_{\lambda_1}, \ldots, y_{\lambda_k})^T$. Hence $(x_1, \ldots, x_k)^T = P^{-1}(y_{\lambda_1}, \ldots, y_{\lambda_k})^T$, so $x_1, \ldots, x_k \in \mathfrak{b}_j$, so $\mathfrak{a} \subseteq \mathfrak{b}_j$.
\end{proof}

The proof of the following lemma uses an argument from \cite{wiegand}. Although we only use the lemma in a special case in this section, we will need it in full generality in the next section. \\

\begin{lem} \label{Wiegand Argument}
    Let $(A,M)$ be a local ring and $I$ an ideal of $A$ such that $A$ is complete with respect to $I$. Let $Q_1 \subsetneq Q_2 \subsetneq Q_3$ be a chain of prime ideals of $A$ such that $I \nsubseteq Q_1$. Then there are $|A/M|^{\aleph_0}$ prime ideals $Q'$ with $Q_1 \subsetneq Q' \subsetneq Q_3$.
\end{lem}

\begin{proof}
Note that, by hypothesis, $A$ is Noetherian and has dimension at least two.  Hence, $|\Spec(A)| \leq |A|$.  Now consider the map $f: A \longrightarrow (A/M, A/M^2, A/M^3, \ldots )$ given by $f(a) = (a + M, a + M^2, a + M^3, \ldots)$.  Since $\bigcap_{i = 1}^{\infty} M^i = (0)$, $f$ is injective.  By Lemma \ref{residue field powers}, $|(A/M, A/M^2, A/M^3, \ldots )| \leq |A/M|^{\aleph_0}$, and so $|A| \leq |A/M|^{\aleph_0}$.  It follows that $|\Spec(A)| \leq |A/M|^{\aleph_0}$.  This shows that there are at most $|A/M|^{\aleph_0}$ prime ideals $Q'$ with $Q_1 \subsetneq Q' \subsetneq Q_3$.

We now show that there are at least $|A/M|^{\aleph_0}$ prime ideals $Q'$ with $Q_1 \subsetneq Q' \subsetneq Q_3$.  Let $\S$ be a full set of coset representatives of $A/M$.
Note that every coset except $0+M$ consists entirely of units, so every element of $\S$ but one is a unit. Moreover, note that the difference between any two elements of $\S$ is a unit.

If the chain $Q_1 \subsetneq Q_2 \subsetneq Q_3$ is not saturated, we can find a saturated chain $Q_1 \subsetneq Q_2' \subsetneq Q_3'$ such that $Q_3' \subseteq Q_3$, and proving the result for the saturated chain will imply it for the original chain. 
Thus we assume $Q_1 \subsetneq Q_2 \subsetneq Q_3$ is saturated. Because $I \nsubseteq Q_1$ and $Q_3 \nsubseteq Q_1$, we can choose $b \in (I \cap Q_3) \backslash Q_1$. Note that $\text{ht}(Q_3/Q_1) \ge 2$ in $A/Q_1$, so $Q_3$ cannot be minimal over $Q_1 + (b)$. Thus we can choose $a \in Q_3 \backslash (\bigcup_{Q \in \Min(Q_1 + (b))} Q)$.
Now, consider an element $z$ of the form:
\[
z = a + q_1b + q_2b^2 + \cdots
\]
for $q_i \in \S$. 
Note that every element of this form is in $A$ because $A$ is complete with respect to $I$ and that the cardinality of the set of such elements is $|A/M|^{\aleph_0}$. Furthermore note that each element of this form is in $Q_3$.
\par Observe that $Q_1 + (z) \subseteq Q_3$ and $\text{ht}(Q_1 + (z)) \leq \text{ht}(Q_1) + 1$, so there exists a minimal prime ideal $Q'$ over $Q_1 + (z)$ such that $Q' \subsetneq Q_3$. If $b \in Q'$, then $a \in Q'$, but $Q'$ is minimal over $Q_1 + (b)$, a contradiction. 
Thus $b \notin Q'$. 
Now, assume that there exist $z,z' \in Q'$ of the above form with $z \neq z'$. Then
\begin{align}
    \notag z - z' &= a - a + (q_1 - q_1')b + (q_2 - q_2')b^2 + \cdots \\
    \notag &= b^k[ (q_k - q_k') + (q_{k + 1}-q_{k + 1}')b + \cdots] \\
    \notag &\in Q'
\end{align}
where we have factored out the highest power of $b$ up to which the coefficients of $z$ and $z'$ agree. 
Since $q_k - q_k'$ is a unit, $(q_k-q_k') + (q_{k+1}-q_{k+1}')b + \cdots$ is also a unit and thus $(q_k-q_k') + (q_{k+1}-q_{k+1}')b + \cdots \notin Q'$. 
But $b \notin Q'$, so $b^k \notin Q'$, a contradiction. Thus every element $z$ is in a distinct prime ideal $Q'$ with $Q_1 \subseteq Q' \subsetneq Q_3$. 
Equality for the first inclusion only holds for one prime ideal, so there are $|A/M|^{\aleph_0}$ prime ideals $Q'$ with $Q_1 \subsetneq Q' \subsetneq Q_3$.
\end{proof}

\begin{rmk} \label{bounds}
Observe that given a complete local ring $(T,M)$ and any precompletion $(A,M \cap A)$ of $T$,  $|T/M|  = |A/(M \cap A)|\leq |A| \leq |T|$. Furthermore, if the dimension of $T$ is at least two, then, by Theorem \ref{Sharp 2.6}, letting $\{\mathfrak{b}_i\}_{i \in \mathcal{I}}$ be the set of height one prime ideals of $A$ and letting $\mathfrak{a}$ be $M \cap A$, we have $|A/(M \cap A)| \leq |\Spec(A)|$. Therefore, in this case, we have $|T/M| = |A/(M \cap A)| \leq |\Spec(A)| \leq |\Spec(T)| \leq |T|$.
It follows that, if the dimension of $T$ is at least two, and if $|T|=|T/M|$, then $|\Spec(A)|= |A|= |T| = |\Spec(T)|$ for all precompletions $A$ of $T$.  
\end{rmk}

\begin{rmk}
For most of this section, we will assume that $\dim(T) \ge 2$. This allows for succinct statements of many theorems, since the case where $\dim(T) < 2$ often needs to be dealt with separately. However, this is not a significant restriction because having dimension one or zero already characterizes the spectrum of a local domain. 
Note that assuming $\dim(T) \ge 2$ will generally require a modification of condition (ii) of Theorem \ref{Lech}, because it cannot be the case that $M = (0)$ under this assumption.
\end{rmk}

By Remark \ref{bounds}, we know that if $T$ is a complete local ring with dimension at least two, and if $|T/M| = |T|$, then $|\Spec(T)| = |T|$.  The following theorem tells us that $|\Spec(T)| = |T|$ without the hypothesis $|T/M| = |T|$. \\

\begin{thm} \label{spec dim2}
    Let $(T, M)$ be a complete local ring with $\textup{dim}(T) \ge 2$. Then $|\Spec(T)| = |T|$, and in particular, $T$ and $\Spec(T)$ are uncountable.
\end{thm}
\begin{proof}
Note that $T$ is complete with respect to $M$, and there must exist a chain $Q \subsetneq P \subsetneq M$ of prime ideals of $T$ because $\textup{dim}(T) \ge 2$. Therefore, $|T/M|^{\aleph_0} \leq |\Spec(T)| \leq |T|$ by Lemma \ref{Wiegand Argument}, but $|T/M|^{\aleph_0} \geq |T|$ by the definition of completion, so $|T/M|^{\aleph_0} = |\Spec(T)| = |T|$.  It follows that $T$ and $\Spec(T)$ are uncountable.
\end{proof}

We now give conditions for a complete local ring $T$ of dimension at least two to be the completion of a domain with a spectrum smaller than that of $T$. \\

\begin{thm}\label{countable spec}
    Let $(T, M)$ be a complete local ring with $\textup{dim}(T) \ge 2$. Then $T$ is the completion of a local domain $A$ with $|\Spec(A)| < |\Spec(T)|$ if and only if the following three conditions hold.
    \begin{enumerate}[label={(\roman*)}]
    \item No integer of $T$ is a zerodivisor
    \item $M \notin \Ass(T)$
    \item $|T/M| < |T|$
    \end{enumerate}
\end{thm}

\begin{proof}
Suppose (i),(ii), and (iii) hold; then the statement follows from Corollary \ref{small domain thm}, which guarantees the existence of a precompletion $A$ such that $|\Spec(A)| \le |A| < |T| \le |\Spec(T)|$.
\par Now suppose that $T$ is the completion of a local domain $A$ with $|\Spec(A)| < |\Spec(T)|$. Then (i) and (ii) hold by Theorem \ref{Lech}, and (iii) follows from Remark \ref{bounds}.
\end{proof}
\begin{thm}\label{countable spec cor}
    Let $(T, M)$ be a complete local ring with $\textup{dim}(T) \ge 2$. Then $T$ is the completion of a local domain with countable spectrum if and only if the following three conditions hold.
    \begin{enumerate}[label={(\roman*)}]
    \item No integer of $T$ is a zerodivisor
    \item $M \notin \Ass(T)$
    \item $T/M$ is countable.
    \end{enumerate}
\end{thm}
\begin{proof}
   If (i), (ii), and (iii) hold, then $T$ is the completion of a countable local domain $A$ by Corollary \ref{countable domain thm}.  Since $A$ is countable and Noetherian, $\Spec(A)$ is also countable.
   
   Now suppose that $T$ is the completion of a local domain $A$ with $\Spec(A)$ countable.  Then (i) and (ii) hold by Theorem \ref{Lech}.  By Remark \ref{bounds}, we have $|A/(M \cap A)| \leq |\Spec(A)|$, and so $|T/M| = |A/(M \cap A)|\leq |\Spec(A)|$.  It follows that $T/M$ is countable.
   \end{proof}

We now consider when a complete local ring $(T, M)$ is the completion of a domain $A$ with $|\Spec(T)| = |\Spec(A)|$ or $|T| = |A|$. Note that all of the following results can be restated less generally in terms of countable or uncountable cardinalities, given that a complete local ring of dimension greater than two is uncountable.
\par First, we state a result about the cardinality of spectra in general.
The following lemma and remark are Lemma 2.3 and Remark 2.4 from \cite{small17}, adapted from \cite{pippa}. \\

\begin{lem} [\cite{small17} Lemma 2.3]\label{2017 2.3}
Let $(T, M)$ be a complete local ring of dimension at least one, and let $G$ be a set of nonmaximal prime ideals of $T$ where $G$ contains the associated prime ideals of $T$ and such that the set of maximal elements of $G$ is finite. Moreover, suppose that if $Q \in \Spec(T)$ with $Q \subseteq P$ for some $P \in G$ then $Q \in G$. Also suppose that, for each prime ideal $P \in G$, $P$ contains no nonzero integers of $T$. Then there exists a local domain $A$ such that:

\begin{enumerate}[label={(\roman*)}]
    \item $\widehat{A} \cong T$,
    \item If $P$ is a nonzero prime ideal of $A$, then $T \otimes_A k(P) \cong k(P)$, where $k(P) = A_P/PA_P$,
     \item $\{P \in \Spec(T) \mid P \cap A = (0) \} = G$,
    \item If $I$ is a nonzero ideal of $A$, then $A/I$ is complete.
    \end{enumerate}
\end{lem}  

\begin{rmk} \label{2017 2.4}
It is noted in \cite{small17} that, for the $T$ and $A$ in Lemma \ref{2017 2.3}, there is a bijection between the nonzero prime ideals of $A$ and the prime ideals of $T$ that are not in $G$.
\end{rmk}
Note that the conditions in the following theorem are identical to those of Theorem \ref{Lech} when the dimension of $T$ is at least two.
\begin{thm} \label{uncountable spec thm}
    Let $(T, M)$ be a complete local ring  with $\dim(T) \ge 2$. Then $T$ is the completion of a local domain $A$ with $|\Spec(A)| = |\Spec(T)|$ if and only if
    \begin{enumerate}[label={(\roman*)}]
    \item No integer of $T$ is a zerodivisor, and
    \item $M \notin \Ass(T)$.
    \end{enumerate}
\end{thm}

\begin{proof}
The forward direction follows from Theorem \ref{Lech}.  To show the backwards direction, we will use Theorem \ref{spec dim2}, Lemma \ref{2017 2.3}, and Remark \ref{2017 2.4}. Let $G = \{ P \in \Spec(T) \mid P \subseteq Q \text{ for some } Q \in \Ass(T) \}$. Then $G$ satisfies the conditions of Lemma \ref{2017 2.3}, so there exists a domain $A$ such that $\widehat{A} = T$ and, by Remark \ref{2017 2.4}, there is a bijection between $\Spec(T) \backslash G$ and $\Spec(A) \backslash \{ 0 \}$. By a similar argument to that in the proof of Theorem \ref{spec dim2}, there are $|T| = |\Spec (T)|$ coheight one prime ideals in $T$. Then note that $\Ass(T)$ is finite and the only elements of $G$ that are coheight one are in $\Ass(T)$. It follows that $\Spec(T) \backslash G$ has cardinality $|\Spec (T)|$, as does $\Spec(A)$.
\end{proof}
Because the conditions of Theorem \ref{uncountable spec thm} are identical to the conditions of Theorem \ref{Lech} when the dimension of $T$ is at least two, any complete local ring $T$ of dimension at least two that is the completion of a local domain is also the completion of some local domain $A$ with $|\Spec(A)| = |\Spec(T)|$. 
The following corollary shows that an analogous statement holds for $T$ being the completion of a local domain $A$ with $|A| = |T|$. \\

\begin{cor} \label{uncountable ring cor}
    Let $(T, M)$ be a complete local ring with $\textup{dim}(T) \ge 2$. Then $T$ is the completion of a local domain $A$ with $|A| = |T|$ if and only if
    \begin{enumerate}[label={(\roman*)}]
    \item No integer of $T$ is a zerodivisor, and
    \item $M \notin \Ass(T)$.
    \end{enumerate}
\end{cor}

\begin{proof}
By Theorem \ref{uncountable spec thm}, $T$ is the completion of a local domain $A$ with $|\Spec(A)| = |T|$ if conditions (i) and (ii) hold. Then $|T| = |\Spec(A)| \leq |A| \leq |T|$, so $|A| = |T|$. Furthermore, by Theorem \ref{Lech}, $T$ is the completion of a domain only if conditions (i) and (ii) hold.
\end{proof}

\begin{ex}
    Let $T = \frac{\mathbb{Q}[[x,y,z]]}{(xyz)}$. Then by Corollary \ref{countable domain thm}, $T$ is the completion of a countable local domain with countable spectrum, and by Theorem \ref{uncountable spec thm}, $T$ is the completion of an uncountable local domain with uncountable spectrum.
\end{ex}

Note that a countable Noetherian ring cannot have an uncountable spectrum, so no countable precompletion of a complete local ring $T$ can have an uncountable spectrum. One open question that remains is: when does $T$ have an uncountable precompletion $A$ where $\Spec(A)$ is countable? While specific examples of uncountable rings of dimension at least two with a countable spectrum exist, the general question is unresolved. \\

\section{Precompletions with Mixed Cardinality Spectra}
We know that a complete local ring can have different precompletions with very different properties. For example, both $\mathbb{Q}[x,y]_{(x,y)}$ and $\mathbb{Q}[[x]][y]_{(x,y)}$ are precompletions of $\mathbb{Q}[[x,y]]$.
In this example, $\mathbb{Q}[x,y]_{(x,y)}$ is countable, whereas $\mathbb{Q}[[x]][y]_{(x,y)}$ is not. The latter precompletion is of particular interest. More generally, we have that 
\[
\frac{k[[x_1, \ldots ,x_n]][x_{n+1}, \ldots ,x_m]_{(x_1, x_2, \ldots ,x_m)}}{I}
\]
is a precompletion of 
\[
\frac{k[[x_1,\ldots,x_n, x_{n+1},\ldots,x_m]]}{I}
\]
for any field $k$ and any ideal $I$ that is generated by polynomials. That is, we can ``mix" the polynomial and power series rings to any desired degree and still have a precompletion of the power series ring. This gives rise to an interesting spectrum diagram in a particular case: let $R = \mathbb{Q}[[x]][y,z]_{(x,y,z)}/(xy)$. Then the partially ordered set $\Spec(R)$ is as follows below.

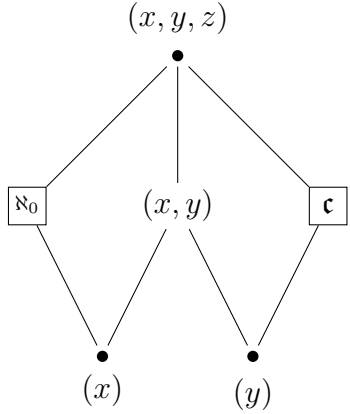
\begin{center}
    \begin{tikzpicture}
        \node[label=above:{$(x,y,z)$}] (A) at (3, 3) {};
        \node[shape=rectangle,draw=black,minimum size=0.5cm,label={[label distance=-0.5cm,align=center]\scriptsize$\aleph_0$}] (B) at (1, 1) {};
        \node[minimum size=0.6cm,label={[label distance=-0.65cm,align=center]$(x,y)$}] (C) at (3, 1) {};
        \node[shape=rectangle,draw=black,minimum size=0.5cm,label={[label distance=-0.5cm,align=center]$\mathfrak{c}$}] (D) at (5, 1) {};
        \node[label={[label distance=-0.95cm,align=center] $(x)$}] (E) at (2, -1) {};
        \node[label=below: $(y)$] (F) at (4, -1) {};
            
        \foreach \n in {A,E,F}
            \node at (\n)[circle,fill,inner sep=1.5pt]{};
        
        \path [-] (A) edge node[left] {} (B);
        \path [-] (A) edge node[left] {} (C);
        \path [-] (A) edge node[left] {} (D);
        \path [-] (B) edge node[left] {} (E);
        \path [-] (C) edge node[left] {} (E);
        \path [-] (C) edge node[left] {} (F);
        \path [-] (D) edge node[left] {} (F);
    \end{tikzpicture}
\end{center}

\hfill \newline
Here the boxes denote collections of prime ideals in the partial order. For instance, there are $\aleph_0$ prime ideals $P$ such that $(x) \subsetneq P \subsetneq (x,y,z)$, and $P$ is incomparable to every other prime ideal. 
It follows that $\hat{R} = \mathbb{Q}[[x,y,z]]/(xy)$ has a precompletion with a spectrum of ``unbalanced cardinality.'' Note that $\hat{R}$ itself does not have a similarly unusual spectrum diagram; in fact, as a consequence of Lemma \ref{Wiegand Argument}, such an unbalance is impossible for any complete local ring.
This observation leads to the question: when does a given complete local ring have a precompletion with a similar ``unbalanced'' spectrum structure?
\par Given a local ring $(A, M)$, we can complete $A$ with respect to an ideal that is not maximal, or even prime. For example, $\mathbb{Q}[[x]][y,z]$ is equal to the completion of $\mathbb{Q}[x,y,z]$ with respect to the $(x)$-adic topology. In the example above, this ``partial completion" gives rise to a ring with an unbalanced spectrum diagram. 
We will prove that in general, we can 
use this technique of creating partial completions to find precompletions with unbalanced cardinalities.
\par First, we verify that most important properties of rings are preserved under partial completions. \\

\begin{lem} \label{quotients preserved}
    Let $(R,M)$ be a local ring and let $I$ and $J$ be ideals of $R$ with $I \subseteq J$.
    Denote by $\hat{R}$ the completion $\varprojlim R/I^n$, and similarly for $\hat{J}$.
    Then we have that $\hat{R}/\hat{J} \cong R/J$.
\end{lem}
\begin{proof}
    This follows from the identities
    \[
    \frac{R}{J} \cong \frac{R/I}{J/I} \cong \frac{\hat{R}/\hat{I}}{\hat{J}/\hat{I}} \cong \frac{\hat{R}}{\hat{J}}
    \]
\end{proof}
\hfill \\

\begin{lem} \label{completion preserved}
    Let $R$ be a Noetherian ring, $n \geq 1$, and $I$ and $J$ ideals of $R$ with $I \subseteq J$. Denote by $\widehat{R}$ and $\widehat{J}$ the completions of $R$ and $J$ with respect to the ideal $I$. Then
    \[
    \varprojlim \widehat{R}/\widehat{J}^n = \varprojlim R/J^n.
    \]
\end{lem}

\begin{proof}
    We have that
    \begin{align*}
        R/J^n
        \cong \frac{R/I^n}{J^n/I^n}
        \cong \frac{\hat{R}/\hat{I}^n}{\hat{J}^n/\hat{I}^n}
        \cong \hat{R}/\hat{J}^n.
    \end{align*}
    Denote the map induced by composing the above natural isomorphisms as $\phi_n$. 
    While the above identities are technically only valid to show that the completions are equivalent as $R$-modules, note that the first and third maps in question are each defined by taking an equivalency class to an equivalency class, such as $r + J^n \mapsto (r + I^n) + J^n/I^n$ (the first map). 
    The middle map is similar; viewing elements in the completion as sequences of ring elements, it takes an equivalency class $[[r]]$ to the equivalency class $[[(r,r,\ldots)]]$, where the double brackets denote that this is an equivalency class whose elements are themselves equivalency classes. 
    Then we have that $\phi(r_1r_2) = \phi(r_1)\phi(r_2)$ because the ring operation on equivalency classes and on sequences in the completion are defined via elementwise multiplication.
    Hence the isomorphisms are not only isomorphisms of modules, but of rings. 
    \par Now consider the diagram
    \begin{center}
        \begin{tikzcd}
        R/J \ar[d, "\phi_1"] & R/J^2 \ar[l] \ar[d, "\phi_2"] & R/J^3 \ar[l] \ar[d, "\phi_3"] & \cdots \ar[l] \\
        \hat{R}/\hat{J} & \hat{R}/\hat{J}^2 \ar[l] & \hat{R}/\hat{J}^3 \ar[l] & \cdots \ar[l] \\
        \end{tikzcd}
    \end{center}
    Note also that the horizontal maps take equivalency classes to equivalency classes. Hence the diagram commutes, as beginning at any object with the element $[r]$ and traversing arrows will always result in being in the equivalency class of $r$. From this we obtain the isomorphism
    \[
    \varprojlim R/J^n \cong \varprojlim \hat{R}/\hat{J}^n
    \]
    and so the desired result holds.
\end{proof}

\par Given a local ring $R$ and prime ideal $P \subseteq R$, we will consider the formal fiber of $P$ as the set of prime ideals $Q \in \Spec(\hat{R})$ such that $Q \cap R = P$ (here $\hat{R}$ denotes $I$-adic completion for some ideal $I$). 
While this is not quite the technical definition of a formal fiber, it is sufficient for our purposes.
The following lemma tells us that when completing a ring with respect to an ideal $I$, the formal fiber of any $P \supseteq I$ is singleton. \\

\begin{lem}\label{primes preserved}
    Let $(R, M)$ be a local ring and let $I$ be an ideal of $R$. Denote by  $\hat{R}$ the completion $\varprojlim R/I^n$. Then the prime ideals of $\hat{R}/\hat{I}$ are exactly those ideals of the form $\hat{Q}/\hat{I}$ for some prime ideal $Q \supseteq I$. 
\end{lem}

\begin{proof}
    This follows from the facts that the isomorphism $R/I \to \hat{R}/\hat{I}$ induces a bijection of prime ideals, and that the prime ideals of $R/I$ are exactly those of the form $Q/I$ for some $Q \supseteq I$.
\end{proof}



\begin{lem} \label{locality preserved}
    Let $(R,M)$ be a local ring and let $I$ be an ideal of $R$. Then $\hat{R} = \varprojlim R/I^n$ is local with maximal ideal $\hat{M}$. 
\end{lem}

\begin{proof}
    By Lemma \ref{primes preserved}, $\hat{M}$ is a maximal ideal of $\hat{R}$. Let $r = ([r_1], [r_2], [r_3], \ldots)$ be an element of $\hat{R} \subset \prod_n R/I^n$. 
    Suppose that $[r_1] \in M/I$. 
    Note that if $[r_2] \not \in M/I^2$, then $r_2 \in r_1 + I$, where $r_2 \not \in M$ and $r_1 \in M$. Since $I \subseteq M$, this is not possible. 
    A similar argument tells us that if $[r_k] \in M/I^k$, then $[r_{k+1}] \in M/I^{k + 1}$. 
    
    Suppose that $r \not \in \hat{M}$. 
    Then $[r_1] \not \in M/I$, else $[r_k] \in M/I^k$ for every $k$ and we would have $r \in \hat{M}$. 
    It follows that each $[r_k] \notin M/I^k$ for all $k$, and therefore $[r_k]$ is a unit. Then without loss of generality, we can let $r_i$ be some unit. 
    Furthermore, the condition $r_j \equiv r_i \, (\text{mod} \, I^i)$ for $j > i$ implies that $r_j^{-1} \equiv r_i^{-1} \, (\text{mod} \, I^i)$ for $j > i$. 
    Hence $([r_1^{-1}], [r_2^{-1}], [r_3^{-1}], \ldots)$ is in $\hat{R}$ and is an inverse for $r$. Thus, $r \not \in \hat{M}$ implies $r$ is a unit, so $\hat{M}$ is the unique maximal ideal of $\hat{R}$. 
\end{proof}

\begin{thm}
    Let $(R,M)$ be a local ring whose $M$-adic completion is $T$. Let $I$ be an ideal of $R$. Denote by ``$\, \hat{} \,$" the $I$-adic completion. Then $(\widehat{R},\widehat{M})$ is a local ring whose $\hat{M}$-adic completion is $T$. 
\end{thm}

\begin{proof}
    This follows from Lemma \ref{completion preserved} (using $J = M$) and from Lemma \ref{locality preserved}.
\end{proof}

As this section is devoted to demonstrating the existence of interesting cardinality structures of spectra of precompletions, we first note that there are some structures which are in general impossible; the following two propositions both tell us that spectra cannot be ``top-heavy.'' \\
\begin{prop}\label{top-heavy ring}
    Let $R$ be a Noetherian ring such that $\Spec(R)$ is infinite.
    Then $|\Spec(R)| = |\{ P \in \Spec(R) \mid \textup{ht}(P) = 1 \}|$. 
\end{prop}

\begin{proof}
    Denote by $\Spec_{\le n}(R)$ the subset of $\Spec(R)$ defined by $\Spec_{\le n}(R) = \{ P \in \Spec(R) \mid \textup{ht}(P) \le n \}$, and by $\Spec_1(R)$ the set of height-$1$ prime ideals. 
    We will show that for each $n \geq 1$, $|\Spec_{\le n}(R)| = |\Spec_1(R)|$.
    This will show that $|\Spec(R)| = |\Spec_1(R)|$.
    We proceed by induction on $n$. The base case $n = 1$ is evident.
    Let $P$ be a height $n + 1$ prime ideal. Let $Q_1, Q_2$ be prime ideals contained in $P$ such that the chains $Q_1 \subsetneq P$ and $Q_2 \subsetneq P$ are saturated. Then $P$ is a minimal prime ideal over $Q_1 + Q_2$. 
    Hence we have that
    \[
    \Spec_{\le n+1}(R) \backslash \Spec_{\le n}(R) \subseteq  \bigcup_{Q_1, Q_2 \in \Spec_{\le n}(R)}\Min(Q_1 + Q_2)
    \]
    Since the number of minimal prime ideals over any given ideal is finite, the number of prime ideals in the union on the right is bounded by $|\Spec_{\le n}(R)|^2 \cdot \aleph_0$. 
    It follows that  $|\Spec_{\le n + 1}(R)| = |\Spec_{\le n}(R)|$. 
    The desired statement follows.
\end{proof}

\begin{prop}\label{top-heavy prime}
    Let $R$ be a Noetherian ring and let $P$ be a prime ideal with $\textup{ht}(P) \ge 2$. Then $|\Spec(R_P)| \ge |R/P|$.
\end{prop}

\begin{proof}
    First, suppose that $R$ is a domain, and let $(0) \subsetneq Q_2 \subsetneq P$ be a chain of prime ideals.
    Consider a nonzero $r \in Q_2$, and $s \in P$ such that $s \not \in Q$ for any $Q \in \Min(r)$. Let $\{u_i\}$ be a full set of coset representatives for $R/P$. 
    Consider the prime ideals of $R$ which are minimal over elements of the form $r + u_is$. 
    Note that two such elements cannot be in the same height one prime ideal $P' \subseteq P$; if $r + u_is$ and $r + u_ks$ were both in $P'$, then we would have $(u_i - u_k)s \in P' \subseteq P$.
    Since $u_i - u_k \not \in P$, then we would have $s \in P'$ and $r \in P'$, which contradicts the fact that $s$ is not in any prime ideals minimal over $r$. 
    Note also that every such element is in some height one prime ideal $P'\subseteq P$, by the principal ideal theorem and the fact that $\text{ht}(P)\ge 2$. 
    Hence, the number of prime ideals $P'$ for which $(0) \subsetneq P' \subsetneq P$ is at least $|R/P|$, which gives the desired result.
    \par In the case that $R$ is not a domain, let $Q_1 \subsetneq Q_2 \subsetneq P$ be a chain of prime ideals. By the domain case above, there are at least $|R/P|$ prime ideals $P'/Q_1$ of $R/Q_1$ such that $(0) \subsetneq P'/Q_1 \subsetneq P/Q_1$, which yields the desired result. 
\end{proof}


Using partial completions and Lemma \ref{Wiegand Argument}, we can now construct rings where certain subsets of the spectrum have different cardinalities. The most familiar example of distinct possible cardinalities is \{countable, uncountable\}, but we state all of our results in full generality. That is, all of the theorems we present apply to cardinalities outside of the countable and uncountable cardinalities as well. \\

\begin{thm}\label{partial completion theorem}
    Let $(R, M)$ be a local ring with $M$-adic completion $T$ and let $I$ be an ideal of $R$.
    Then there exists a faithfully flat local extension ring of $(S, N)$ of $R$ whose completion at $N$ is also equal to $T$, such that $S/IS \cong R/I$ (and in particular $\Spec(S/IS) \cong \Spec(R/I)$) and such that, for any chain $Q_1 \subsetneq Q_2 \subsetneq Q_3$ of prime ideals of $S$ such that $IS \nsubseteq Q_1$, there are $|\Spec(T)|$ prime ideals $Q'$ with $Q_1 \subsetneq Q' \subsetneq Q_3$.
\end{thm}

\begin{proof}
    Let $S =\hat{R}$ be the completion of $R$ at $I$. The first claim follows from the isomorphism $\hat{R}/\hat{I} \cong R/I$. The second follows from Lemma \ref{Wiegand Argument} and the fact that $|S/N|^{\aleph_0} = |T/MT|^{\aleph_0} = |\Spec(T)|$. The equality $|T/MT|^{\aleph_0} = |\Spec(T)|$ itself follows from Lemma \ref{Wiegand Argument}.
\end{proof}

\begin{thm}\label{partial completion avoidance}
    Let $(R, M)$ be a local ring with $\textup{dim}(R) \ge 2$ and $M$-adic completion $T$, and $P_1, \ldots ,P_n$ incomparable prime ideals of $R$.
    Then there exists a faithfully flat extension ring $(S, N)$ whose completion at $N$ is also equal to $T$, such that $S/P_iS \cong R/P_i$ and such that for any chain $Q_1 \subsetneq Q_2 \subsetneq Q_3$ of prime ideals of $S$ such that $P_iS \nsubseteq Q_1$ for all $i$, there are $|\Spec(T)|$ prime ideals $Q'$ with $Q_1 \subsetneq Q' \subsetneq Q_3$. In particular, if $|\Spec(R/P_i)| < |\Spec(T)|$, then $|\Spec(S/P_iS)| < |\Spec(T)|$. 
\end{thm}
\begin{proof}
    Apply Theorem \ref{partial completion theorem} with $I = P_1 \cap \ldots \cap P_n$ and note that the arguments of that theorem still apply for each $P_i$: the identity $S/P_iS \cong R/P_i$ follows from Lemma \ref{quotients preserved}. 
    Note that if $P_iS \not \subseteq Q_1$ for all $i$, then there exist elements $x_1, \ldots ,x_n$ of $R$ such that $x_i \in P_i$ and $x_i \not\in Q_1$ for every $i$.  Therefore $x = \prod_{i = 1}^n x_i \in IS$ and $x \not\in Q_1$, and so $IS \not\subseteq Q_1$.
    \end{proof}

The next example demonstrates a particular application of this theorem.

\begin{ex}\label{domain partial completion example}
Let $T = \mathbb{Q}[[x,y,z,w]]/(xy)$. By Corollary \ref{countable domain thm}, $T$ has a precompletion $A$ which is a countable domain. 
Since $\text{dim}(T) = 3$ and it is equidimensional, we may pick $n$ height one (and therefore coheight two) prime ideals $P_1, \ldots, P_n$ in $A$ and denote by $R$ the completion of $A$ at $I = P_1\cap \ldots \cap P_n$; note $R$ is also a precompletion of $T$. 
This precompletion has the property that there are exactly $n$ prime ideals $P$ of height one such that $|\Spec(R/P)| = \aleph_0$, these being the ideals $P_iR$.
For all other height-one prime ideals $P$, $|\Spec(R/P)| =  \aleph_0^{\aleph_0} = \mathfrak{c}$ by Lemma \ref{Wiegand Argument}.
\end{ex}

\par Aside from Theorem \ref{partial completion avoidance}, partial completions actually have more utility outside of the domain case: it turns out we can use them to find a large class of precompletions with ``unbalanced spectrum,'' answering the question posed at the beginning of this section.
We first need an appropriate precompletion to start with: in \cite{monster}, the authors are able to create a precompletion with cardinality $\sup(\aleph_0, |T/M|)$\footnote{This is basically just shorthand for: ``$|T/M|$ when $T/M$ is infinite, and $\aleph_0$ when $T/M$ is finite.''} and such that the minimal formal fibers are prescribed, for ``most" complete local rings $T$.
This lends itself to our most interesting use of partial completions. \\

\begin{thm} \label{monster and partial}
    Let $(T,M)$ be a complete local ring with $\dim(T) \ge 2$ and $|T/M| < |T|$.
    Suppose that $T$ has exactly $n$ minimal prime ideals $\{Q_1, \ldots, Q_n\}$, and let $\{\C_1,...,\C_m\}$ be a partition of the minimal prime ideals (that is, each $\C_i$ is disjoint from $\C_j$ for $i\neq j$ and  $\bigcup_{i \le m} \C_i = \Min(T)$). 
    Suppose this partition also satisfies the conditions of \monster in \cite{monster} (as a `minfeasible partition').
    Let $k \le m$ be a natural number, and furthermore suppose that $\C_i = \{Q_i\}$ for $i \le k$ and that each $\C_i$ contains a minimal prime ideal $Q$ such that $\cht(Q) \ge 2$.
    Then there is a local ring $R$ with $\hat{R} = T$, such that $R$ has exactly $m$ minimal prime ideals and such that $k$ of these minimal prime ideals are contained in $|\Spec(T)|$ prime ideals, while the other $m - k$ are contained in $\sup(\aleph_0, |T/M|)$ prime ideals.
\end{thm}

\begin{proof}
    Use Theorem 3.19 in \cite{monster} to construct a subring $(A,A \cap M)$ of $T$ with $(M \cap A)$-adic completion $T$ such that $|A| = \sup(\aleph_0, |T/M|)$ and $\Min(A) = \{ q_1,...,q_m\}$, and the elements of $\C_i$ are in the formal fiber of $q_i$. 
    Note that the formal fibers of $q_1,...,q_k$ each have one minimal element. 
    Then $A$ has $m$ minimal prime ideals, all of coheight at least $2$, and $|\Spec(A)| = \sup(\aleph_0, |T/M|)$ by Lemma \ref{Sharp 2.6} (and ordinary prime avoidance, if $T/M$ is finite/countable).
    \par Then use Theorem \ref{partial completion avoidance} with the set of incomparable prime ideals $ \{q_{k+1}, \ldots , q_m\}$ to produce a faithfully flat local extension ring $R$ of $A$ such that $\hat{R} = T$.
    Observe that the completion map $A \to T$ factors through $R$.
    Then, because the formal fiber of the ideals $q_1,...,q_k$ with respect to the map $A \to T$ has one minimal element, the formal fiber of each with respect to the map $A \to R$ has one minimal element. 
    The formal fiber of each of the ideals $q_{k+1},...,q_m$ is singleton (with respect to the map $A \to R$), by Lemma \ref{primes preserved}.
    Therefore, $R$ has exactly $m$ minimal prime ideals. 
    The statement about the cardinalities of the varieties follows from Theorem \ref{partial completion avoidance} and the fact that $|\Spec(A)| = \sup(\aleph_0, |T/M|)$.
\end{proof}
Theorem \ref{monster and partial} tells us that given a complete local ring $(T,M)$ such that $\dim(T) \ge 2$, $|T/M| < |T|$, T has at least $m$ minimal prime ideals of coheight $\ge 2$, and some partition of these satisfies the (fairly weak) conditions of \monster in \cite{monster}, $T$ has a precompletion $R$ such that we can order-embed the following diagram into $\Spec(R)$:
\newpage 
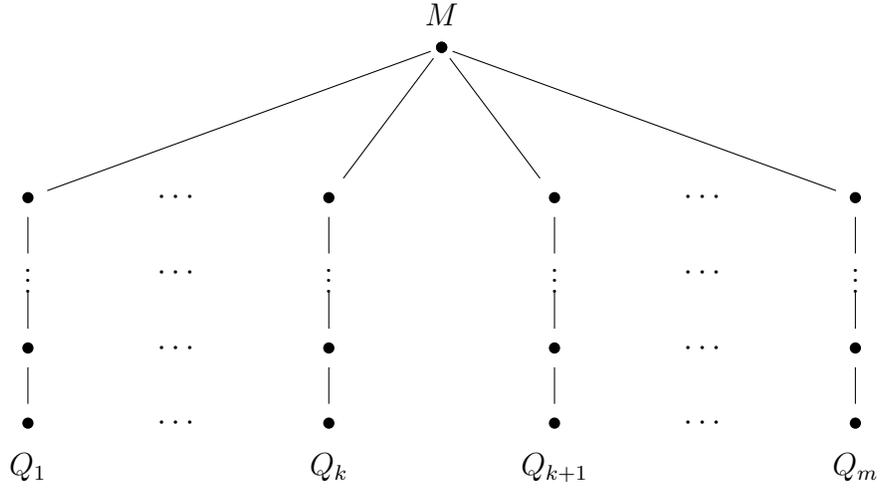
{\centering
    \hspace{.75cm}
    \begin{figure*}[h!]
        \centering
        \begin{tikzpicture}[remember picture]
            \node[label=above:$M$] (max) at (6.5, 3) {};
            \node[minimum size=0.5cm] (A1) at (1, 1) {};
            \node[minimum size=0.5cm,label=center:$\cdots$] (dots) at (3, 1) {};
            \node[minimum size=0.5cm] (B1) at (5, 1) {};
            
            \node[minimum size=0.5cm,label=center:$\vdots$] (Adots) at (1, 0) {};
            \node[minimum size=0.5cm,label=center:$\cdots$] (dots) at (3, 0) {};
            \node[minimum size=0.5cm,label=center:$\vdots$] (Bdots) at (5, 0) {};
            
            \node[minimum size=0.5cm] (A2) at (1, -1) {};
            \node[minimum size=0.5cm,label=center:$\cdots$] (dots) at (3, -1) {};
            \node[minimum size=0.5cm] (B2) at (5, -1) {};
            
            \node[minimum size=0.5cm,label=below:$Q_1$] (Amin) at (1, -2) {};
            \node[minimum size=0.5cm,label=center:$\cdots$] (dots) at (3, -2) {};
            \node[minimum size=0.5cm,label=below:$Q_k$] (Bmin) at (5, -2) {};
            
            \foreach \n in {max, A1, A2, Amin, B1, B2, Bmin}
                \node at (\n)[circle,fill,inner sep=1.5pt]{};
            
            \foreach \n in {A1,B1}
                \path [-] (\n) edge node[left] {} (max);
            
            \path [-] (A1) edge node[left] {} (Adots);
            \path [-] (Adots) edge node[left] {} (A2);
            \path [-] (A2) edge node[left] {} (Amin);
            
            \path [-] (B1) edge node[left] {} (Bdots);
            \path [-] (Bdots) edge node[left] {} (B2);
            \path [-] (B2) edge node[left] {} (Bmin);
                
            \node[minimum size=0.5cm] (A1) at (8, 1) {};
            \node[minimum size=0.5cm,label=center:$\cdots$] (dots) at (10, 1) {};
            \node[minimum size=0.5cm] (B1) at (12, 1) {};
            
            \node[minimum size=0.5cm,label=center:$\vdots$] (Adots) at (8, 0) {};
            \node[minimum size=0.5cm,label=center:$\cdots$] (dots) at (10, 0) {};
            \node[minimum size=0.5cm,label=center:$\vdots$] (Bdots) at (12, 0) {};
            
            \node[minimum size=0.5cm] (A2) at (8, -1) {};
            \node[minimum size=0.5cm,label=center:$\cdots$] (dots) at (10, -1) {};
            \node[minimum size=0.5cm] (B2) at (12, -1) {};
            
            \node[minimum size=0.5cm,label=below:$Q_{k+1}$] (Amin) at (8, -2) {};
            \node[minimum size=0.5cm,label=center:$\cdots$] (dots) at (10, -2) {};
            \node[minimum size=0.5cm,label=below:$Q_{m}$] (Bmin) at (12, -2) {};
            
            \foreach \n in {max, A1, A2, Amin, B1, B2, Bmin}
                \node at (\n)[circle,fill,inner sep=1.5pt]{};
            
            \foreach \n in {A1,B1}
                \path [-] (\n) edge node[left] {} (max);
            
            \path [-] (A1) edge node[left] {} (Adots);
            \path [-] (Adots) edge node[left] {} (A2);
            \path [-] (A2) edge node[left] {} (Amin);
            
            \path [-] (B1) edge node[left] {} (Bdots);
            \path [-] (Bdots) edge node[left] {} (B2);
            \path [-] (B2) edge node[left] {} (Bmin);

        \end{tikzpicture}
        \label{fig:spect}
    \end{figure*}
}
\newline
and such that if $P_1 \subsetneq P_2 \subsetneq P_3$, in the above diagram, then if $P_1$ is on the right side (contains one of $Q_{k+1}, \ldots ,Q_m)$ there are $\sup(\aleph_0, |T/M|)$ prime ideals $P'$ in $R$ with $P_1 \subsetneq P' \subsetneq P_3$, and otherwise (that is, if $P_1$ contains one of $Q_{1}, \ldots ,Q_k$) there are $|\Spec(T)|$
prime ideals $P'$ in $R$ with $P_1 \subsetneq P' \subsetneq P_3$. Essentially, the right side is much smaller than the left side (we take this to be a diagram such that prime ideals of different chains are incomparable). 

\par In \cite{monster}, the authors also find an upper bound on the cardinality of $\Min(Q_1 + Q_2)$ for the constructed subring $A$, and show that $A$ can be constructed to meet those upper bounds, giving further control over the spectrum as a partially ordered set. In the following example, we use the results from this paper and from \cite{monster} to find the known possible spectra of precompletions of a particular complete local ring. \\
\newpage
\begin{ex}\label{standard}
    Let $p_1 = x+y$, $p_2 = y+z$, and $p_3 = z + x$. Let T = $\mathbb{Q}[[x,y,z]]/(p_1p_2p_3)$. The partially ordered set $\Spec(T)$ is as follows: 
    
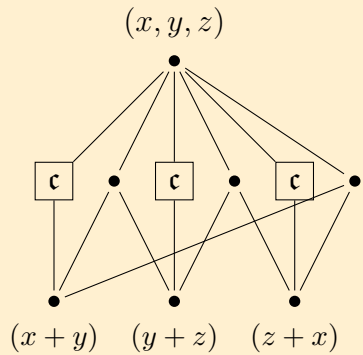
\begin{center}
    \begin{tikzpicture}[scale = .8]
        \node[label=above:{$(x,y,z)$}] (top) at (9, 3) {};
        
        \node[shape=rectangle,draw=black,minimum size=0.5cm,label=center:{$\mathfrak{c}$}] (mid1) at (7, 1) {};
        \node[shape=rectangle,draw=black,minimum size=0.5cm,label=center:{$\mathfrak{c}$}] (mid2) at (9, 1) {};    \node[shape=rectangle,draw=black,minimum size=0.5cm,label=center:{$\mathfrak{c}$}] (mid3) at (11, 1) {};
        
        \node[] (shared1) at (8,1) {};
        \node[] (shared2) at (10,1) {};
        \node[] (shared3) at (12,1) {};
        
        \node[label=below:{\small$( x + y )$}] (min1) at (7, -1) {};
        \node[label=below:{\small$( y + z)$}] (min2) at (9, -1) {};
        \node[label=below:{\small$( z + x )$}] (min3) at (11,-1) {};
        
        \foreach \n in {top,shared1,shared2,shared3,min1,min2,min3}
            \node at (\n)[circle,fill,inner sep=1.5pt]{};
        
        \foreach \n in {shared1,shared2,shared3,mid1,mid2,mid3}
        \path [-] (top) edge node[left] {} (\n);
        
        \path [-] (mid1) edge node[left] {} (min1);
        \path [-] (mid2) edge node[left] {} (min2);
        \path [-] (mid3) edge node[left] {} (min3);
        
        \path [-] (shared1) edge node[left] {} (min1);
        \path [-] (shared1) edge node[left] {} (min2);
        
        \path [-] (shared2) edge node[left] {} (min2);
        \path [-] (shared2) edge node[left] {} (min3);
        
        \path [-] (shared3) edge node[left] {} (min3);
        \path [-] (shared3) edge node[left] {} (min1);
    \end{tikzpicture}
\end{center}

Using Theorem \ref{monster and partial}, along with the results in \cite{monster}, we can construct precompletions of $T$ which have the following prime spectra:
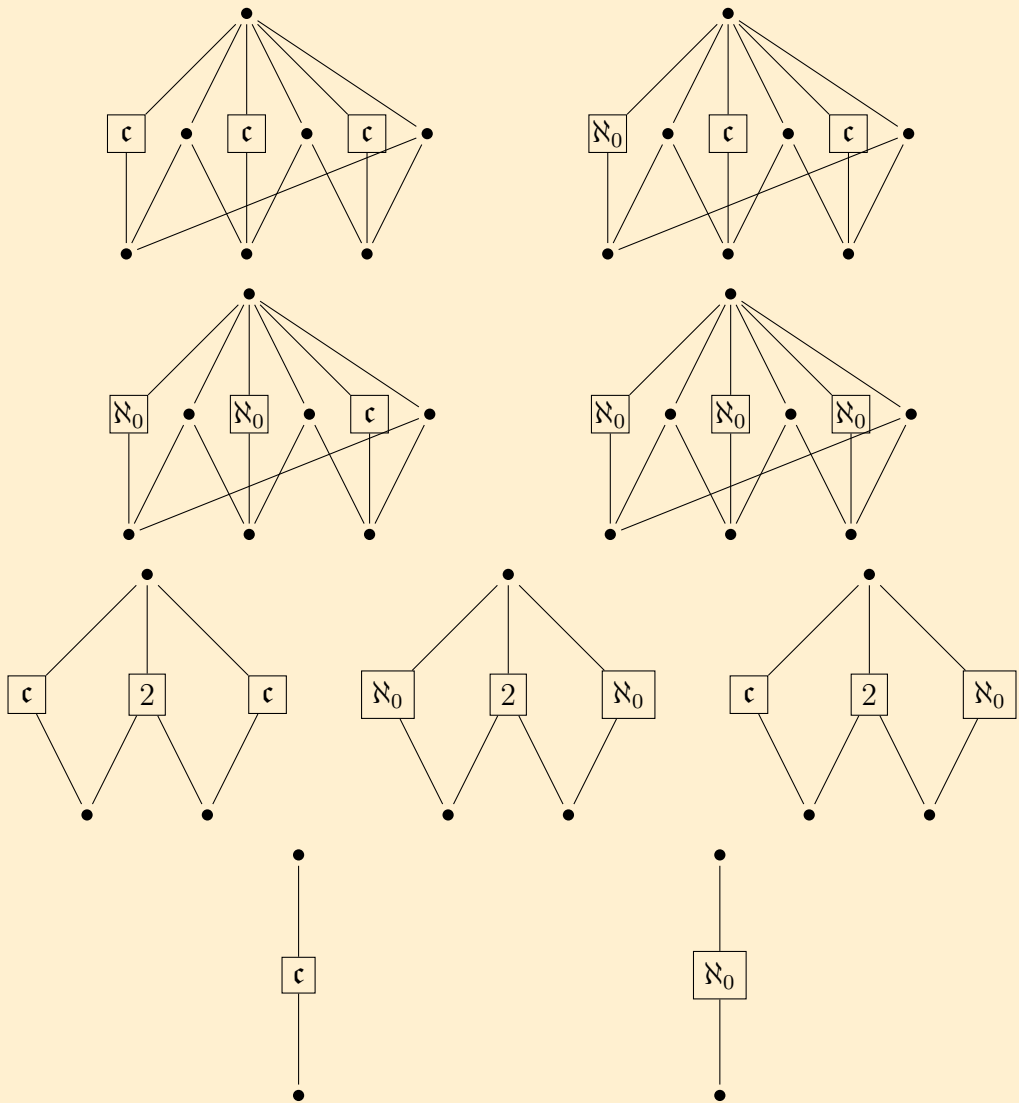
\begin{center}
    \begin{tikzpicture}[scale = .8]
        \node[] (top) at (9, 3) {};
        
        \node[shape=rectangle,draw=black,minimum size=0.5cm,label=center:{$\mathfrak{c}$}] (mid1) at (7, 1) {};
        \node[shape=rectangle,draw=black,minimum size=0.5cm,label=center:{$\mathfrak{c}$}] (mid2) at (9, 1) {};    \node[shape=rectangle,draw=black,minimum size=0.5cm,label=center:{$\mathfrak{c}$}] (mid3) at (11, 1) {};
        
        \node[] (shared1) at (8,1) {};
        \node[] (shared2) at (10,1) {};
        \node[] (shared3) at (12,1) {};
        
        \node[] (min1) at (7, -1) {};
        \node[] (min2) at (9, -1) {};
        \node[] (min3) at (11,-1) {};
        
        \foreach \n in {top,shared1,shared2,shared3,min1,min2,min3}
            \node at (\n)[circle,fill,inner sep=1.5pt]{};
        
        \foreach \n in {shared1,shared2,shared3,mid1,mid2,mid3}
        \path [-] (top) edge node[left] {} (\n);
        
        \path [-] (mid1) edge node[left] {} (min1);
        \path [-] (mid2) edge node[left] {} (min2);
        \path [-] (mid3) edge node[left] {} (min3);
        
        \path [-] (shared1) edge node[left] {} (min1);
        \path [-] (shared1) edge node[left] {} (min2);
        
        \path [-] (shared2) edge node[left] {} (min2);
        \path [-] (shared2) edge node[left] {} (min3);
        
        \path [-] (shared3) edge node[left] {} (min3);
        \path [-] (shared3) edge node[left] {} (min1);
        
        
        \node[] (top) at (17, 3) {};
        
        \node[shape=rectangle,draw=black,minimum size=0.5cm,label=center:{$\aleph_0$}] (mid1) at (15, 1) {};
        \node[shape=rectangle,draw=black,minimum size=0.5cm,label=center:{$\mathfrak{c}$}] (mid2) at (17, 1) {};    \node[shape=rectangle,draw=black,minimum size=0.5cm,label=center:{$\mathfrak{c}$}] (mid3) at (19, 1) {};
        
        \node[] (shared1) at (16,1) {};
        \node[] (shared2) at (18,1) {};
        \node[] (shared3) at (20,1) {};
        
        \node[] (min1) at (15, -1) {};
        \node[] (min2) at (17, -1) {};
        \node[] (min3) at (19,-1) {};
        
        \foreach \n in {top,shared1,shared2,shared3,min1,min2,min3}
            \node at (\n)[circle,fill,inner sep=1.5pt]{};
        
        \foreach \n in {shared1,shared2,shared3,mid1,mid2,mid3}
        \path [-] (top) edge node[left] {} (\n);
        
        \path [-] (mid1) edge node[left] {} (min1);
        \path [-] (mid2) edge node[left] {} (min2);
        \path [-] (mid3) edge node[left] {} (min3);
        
        \path [-] (shared1) edge node[left] {} (min1);
        \path [-] (shared1) edge node[left] {} (min2);
        
        \path [-] (shared2) edge node[left] {} (min2);
        \path [-] (shared2) edge node[left] {} (min3);
        
        \path [-] (shared3) edge node[left] {} (min3);
        \path [-] (shared3) edge node[left] {} (min1);
    \end{tikzpicture}
\end{center}

\begin{center}
    \begin{tikzpicture}[scale = .8]
        \node[] (top) at (9, 3) {};
        
        \node[shape=rectangle,draw=black,minimum size=0.5cm,label=center:{$\aleph_0$}] (mid1) at (7, 1) {};
        \node[shape=rectangle,draw=black,minimum size=0.5cm,label=center:{$\aleph_0$}] (mid2) at (9, 1) {};    \node[shape=rectangle,draw=black,minimum size=0.5cm,label=center:{$\mathfrak{c}$}] (mid3) at (11, 1) {};
        
        \node[] (shared1) at (8,1) {};
        \node[] (shared2) at (10,1) {};
        \node[] (shared3) at (12,1) {};
        
        \node[] (min1) at (7, -1) {};
        \node[] (min2) at (9, -1) {};
        \node[] (min3) at (11,-1) {};
        
        \foreach \n in {top,shared1,shared2,shared3,min1,min2,min3}
            \node at (\n)[circle,fill,inner sep=1.5pt]{};
        
        \foreach \n in {shared1,shared2,shared3,mid1,mid2,mid3}
        \path [-] (top) edge node[left] {} (\n);
        
        \path [-] (mid1) edge node[left] {} (min1);
        \path [-] (mid2) edge node[left] {} (min2);
        \path [-] (mid3) edge node[left] {} (min3);
        
        \path [-] (shared1) edge node[left] {} (min1);
        \path [-] (shared1) edge node[left] {} (min2);
        
        \path [-] (shared2) edge node[left] {} (min2);
        \path [-] (shared2) edge node[left] {} (min3);
        
        \path [-] (shared3) edge node[left] {} (min3);
        \path [-] (shared3) edge node[left] {} (min1);
        
        
        \node[] (top) at (17, 3) {};
        
        \node[shape=rectangle,draw=black,minimum size=0.5cm,label=center:{$\aleph_0$}] (mid1) at (15, 1) {};
        \node[shape=rectangle,draw=black,minimum size=0.5cm,label=center:{$\aleph_0$}] (mid2) at (17, 1) {};    \node[shape=rectangle,draw=black,minimum size=0.5cm,label=center:{$\aleph_0$}] (mid3) at (19, 1) {};
        
        \node[] (shared1) at (16,1) {};
        \node[] (shared2) at (18,1) {};
        \node[] (shared3) at (20,1) {};
        
        \node[] (min1) at (15, -1) {};
        \node[] (min2) at (17, -1) {};
        \node[] (min3) at (19,-1) {};
        
        \foreach \n in {top,shared1,shared2,shared3,min1,min2,min3}
            \node at (\n)[circle,fill,inner sep=1.5pt]{};
        
        \foreach \n in {shared1,shared2,shared3,mid1,mid2,mid3}
        \path [-] (top) edge node[left] {} (\n);
        
        \path [-] (mid1) edge node[left] {} (min1);
        \path [-] (mid2) edge node[left] {} (min2);
        \path [-] (mid3) edge node[left] {} (min3);
        
        \path [-] (shared1) edge node[left] {} (min1);
        \path [-] (shared1) edge node[left] {} (min2);
        
        \path [-] (shared2) edge node[left] {} (min2);
        \path [-] (shared2) edge node[left] {} (min3);
        
        \path [-] (shared3) edge node[left] {} (min3);
        \path [-] (shared3) edge node[left] {} (min1);
    \end{tikzpicture}
\end{center}

\begin{center}
    \begin{tikzpicture}[scale = 0.8]
        \node[] (A) at (9, 3) {};
        \node[shape=rectangle,draw=black,minimum size=0.5cm] (B) at (7, 1) {$\aleph_0$};
        \node[shape=rectangle,draw=black] (C) at (9, 1) {$2$};            \node[shape=rectangle,draw=black,minimum size=0.5cm] (D) at (11, 1) {$\aleph_0$};
        \node[] (E) at (8, -1) {};
        \node[] (F) at (10, -1) {};
            
        \foreach \n in {A,E,F}
            \node at (\n)[circle,fill,inner sep=1.5pt]{};
        
        \path [-] (A) edge node[left] {} (B);
        \path [-] (A) edge node[left] {} (C);
        \path [-] (A) edge node[left] {} (D);
        \path [-] (B) edge node[left] {} (E);
        \path [-] (C) edge node[left] {} (E);
        \path [-] (C) edge node[left] {} (F);
        \path [-] (D) edge node[left] {} (F);
        
        
        \node[] (A) at (15, 3) {};
        \node[shape=rectangle,draw=black,minimum size=0.5cm] (B) at (13, 1) {$\mathfrak{c}$};
        \node[shape=rectangle,draw=black] (C) at (15, 1) {$2$};            \node[shape=rectangle,draw=black,minimum size=0.5cm] (D) at (17, 1) {$\aleph_0$};
        \node[] (E) at (14, -1) {};
        \node[] (F) at (16, -1) {};
            
        \foreach \n in {A,E,F}
            \node at (\n)[circle,fill,inner sep=1.5pt]{};
        
        \path [-] (A) edge node[left] {} (B);
        \path [-] (A) edge node[left] {} (C);
        \path [-] (A) edge node[left] {} (D);
        \path [-] (B) edge node[left] {} (E);
        \path [-] (C) edge node[left] {} (E);
        \path [-] (C) edge node[left] {} (F);
        \path [-] (D) edge node[left] {} (F);
        
        
        \node[] (A) at (3, 3) {};
        \node[shape=rectangle,draw=black,minimum size=0.5cm] (B) at (1, 1) {$\mathfrak{c}$};
        \node[shape=rectangle,draw=black] (C) at (3, 1) {$2$};            \node[shape=rectangle,draw=black,minimum size=0.5cm] (D) at (5, 1) {$\mathfrak{c}$};
        \node[] (E) at (2, -1) {};
        \node[] (F) at (4, -1) {};
            
        \foreach \n in {A,E,F}
            \node at (\n)[circle,fill,inner sep=1.5pt]{};
        
        \path [-] (A) edge node[left] {} (B);
        \path [-] (A) edge node[left] {} (C);
        \path [-] (A) edge node[left] {} (D);
        \path [-] (B) edge node[left] {} (E);
        \path [-] (C) edge node[left] {} (E);
        \path [-] (C) edge node[left] {} (F);
        \path [-] (D) edge node[left] {} (F);
    \end{tikzpicture}
\end{center}

\begin{center}
    \begin{tikzpicture}[scale = 0.8]
        \node[] (A) at (5, 3) {};

        \node[shape=rectangle,draw=black] (C) at (5, 1) {$\mathfrak{c}$};        
        \node[] (E) at (5, -1) {};
            
        \foreach \n in {A,E}
            \node at (\n)[circle,fill,inner sep=1.5pt]{};
        
        \path [-] (A) edge node[left] {} (C);
        \path [-] (C) edge node[left] {} (E);
        
        
        \node[] (A) at (12, 3) {};

        \node[shape=rectangle,draw=black] (C) at (12, 1) {$\aleph_0$};        
        \node[] (E) at (12, -1) {};
            
        \foreach \n in {A,E}
            \node at (\n)[circle,fill,inner sep=1.5pt]{};
        
        \path [-] (A) edge node[left] {} (C);
        \path [-] (C) edge node[left] {} (E);
    \end{tikzpicture}
\end{center}

\end{ex}
The precompletion spectra for the complete local ring in Example \ref{standard} which are known as a result of this paper are those of ``unbalanced" cardinality, though we include all known precompletion spectra, including those obtained in \cite{monster}.
The only possible, but unconfirmed prime spectra for this example are those in the third row of Example \ref{standard} with the boxes containing ``2" replaced with ``1".\footnote{Of course, there is also the unconfirmed possibility of having one of the boxes be cardinality $\alpha$ with $\aleph_0 < \alpha < \mathfrak{c}$. For obvious reasons, (namely the independence of ZFC of whether such cardinals even exist) we do not try to construct any such precompletions. Note that in general the disbalance we find in this paper is always between cardinalities $\beta$ and $\beta^{\aleph_0}$, where $\beta = \sup(\aleph_0, |T/M|)$. This is no coincidence - one corollary of the Generalized Continuum Hypothesis would be that there are no cardinalities between these, so looking for them in precompletions would be a useless endeavor.}
\par Note that an unbalanced cardinality structure is impossible in the case that $|T/M| = |T/M|^{\aleph_0} = |T|$, as per the argument in Remark \ref{bounds}. Hence the condition $|T/M| < |T|$ of Theorem \ref{monster and partial} is necessary. 
A question still left unanswered is whether there are cases outside the scope of \monster of \cite{monster} such that an unbalanced precompletion of this sort with a prescribed number of minimal prime ideals can be found. 
However, we can provide necessary and sufficient conditions for a complete local ring to be the completion of some local ring with unbalanced spectrum, even though we cannot always control the minimal primes well.
\hfill \\
\begin{thm} \label{unbalanced spec precompletion}
    Let $T$ be a complete local ring. 
    Then $T$ is the completion of a local ring $R$ such that there exists $Q_1,Q_2 \in \Spec(R)$ with $Spec(R/Q_1)$ and $Spec(R/Q_2)$ infinite and $|\Spec(R/Q_1)| \neq |\Spec(R/Q_2)|$ if and only if both of the following conditions are satisfied:
    \begin{enumerate}[label={(\roman*)}]
        \item $\Spec(T)$ contains at least $2$ distinct prime ideals of coheight $\ge 2$. 
        \item $|T/M| < |T|$
    \end{enumerate}
\end{thm}

\begin{proof}
    We first prove these conditions are necessary: let $R$ be a precompletion of $T$, and $Q_1,Q_2 \in \Spec(R)$ with $|\Spec(R/Q_1)| < |\Spec(R/Q_2)|$ (and both cardinalities infinite). 
    Since both cardinalities are infinite, both $Q_1$ and $Q_2$ have coheight $\ge 2$. Therefore, some element of the formal fiber of each must have coheight $\ge 2$, so condition (i) holds.
    Also, note that as in Remark \ref{bounds}, $|T/M| \le |\Spec(R/Q_1)| < |\Spec(R/Q_2)| \le |T|$ and therefore $|T/M| < |T|$ so condition (ii) is satisfied. 
    \par Now let $T$ be a complete local ring such that conditions (i) and (ii) hold. 
    The proof hinges on the fact that if $\Min(T) = \{Q_1,...,Q_n\}$, then the partition $\{\{Q_1\},...,\{Q_n\}\}$ is a minfeasible partition that always satisfies the conditions of \monster in \cite{monster},
    so that by Theorem 3.19 in \cite{monster}, there exists a precompletion $R$ of $T$ of cardinality $\sup(\aleph_0, |T/M|)$ which has $n$ distinct minimal prime ideals $\{Q_1 \cap R,...,Q_n \cap R\}$.
    We now consider two cases.
    \par First, suppose that there is only one minimal prime ideal of $T$ with coheight $\ge 2$, which we denote by $Q$. 
    Then $\dim(T) \ge 3$ as a consequence of condition (i).
    Recall that $R$ from the previous paragraph satisfies $|R| = \sup(\aleph_0, |T/M|).$
    Let $P \in \Spec(R)$ be any coheight $2$ prime ideal, and apply Theorem \ref{partial completion theorem} with $I = P$, producing another precompletion $A$ as the $P$-adic completion of $R$. 
    Note that $\cht(Q \cap R) \ge 3$, so in particular $Q \cap R \not \supseteq P$. Then we have that $|\Spec(A/(Q\cap R)A)| = |\Spec(T)| = |T|$, and $|\Spec(A/PA)| = |\Spec(R/P)| = \sup(\aleph_0, |T/M|)$. Thus in the case that there is exactly one minimal prime ideal in $T$ of coheight $\ge 2$, we can construct a precompletion as desired. 
    \par Now suppose that there exist at least $2$ minimal prime ideals in $Q_1,Q_2 \in \Spec(T)$ such that each has coheight $\ge 2$.
    Apply Theorem \ref{partial completion theorem} with $I = Q_1 \cap R$, producing a precompletion $A$ as the $(Q_1 \cap R)$-adic completion of $R$. 
    Then we have that $|\Spec(A/(Q_2 \cap R)A)| = |\Spec(T)| = |T|$, and $|\Spec(A/(Q_1 \cap R)A)| = |\Spec(R/(Q_1 \cap R))| = \sup(\aleph_0, |T/M|)$. Thus in the case that there are at least two minimal prime ideals of $T$ of coheight $\ge 2$, we can construct a precompletion as desired. 
\end{proof}

\par Note that if $T$ contains multiple minimal prime ideals of coheight $\ge 2$, then by the construction in Theorem \ref{unbalanced spec precompletion}, we can embed a diagram as the one in Theorem \ref{monster and partial} into the spectrum of some precompletion $R$. However, our ability to control the number of minimal prime ideals in this precompletion hinges on our ability to find a suitable minfeasible partition.
\par Furthermore, note that we are not constrained to apply this technique only to minimal prime ideals. 
The most general form, Theorem \ref{partial completion avoidance}, gives us the ability to essentially pick the union of finitely many varieties, defined by any $n$ prime ideals $P_1, \ldots ,P_n$, and by completing at the ideal $P_1 \cap \ldots \cap P_n$, preserve the spectrum of the ring within those varieties, while (in the interesting case where we begin with a precompletion $A$ such that $|\Spec(A)| < |\Spec(T)|$) increasing the number of prime ideals everywhere else. Some of these may be minimal, or perhaps none may be. 
\par This also gives a quick example to show that the reverse inequality of Proposition \ref{top-heavy prime} is not true, and that in fact counterexamples are abundant, in that we can take any local countable ring of dimension $\ge 4$ and by completing at a prime ideal $P$ with $\text{ht}(P) = \text{coht}(P) = 2$, produce a local ring extension such that the image of $P$ contains uncountably many prime ideals, but is contained in only countably many---and such that this new ring will have the same completion as the old (with respect to the maximal ideal). 
Thus, while Propositions 4.6 and 4.7 tell us ``top-heavy" spectra are impossible, many complete local rings have precompletions whose spectra are ``bottom-heavy."
\par Of course, one advantage to completing with respect to an intersection of minimal prime ideals is that it is guaranteed that the minimal prime ideals in the intersection don't ``split"---that is, that their formal fibers are a single element (note that in Example \ref{domain partial completion example}, we were unable to say anything about the partial completion being a domain). 
When performing a partial completion on a precompletion, this can be used (as it is in Theorem \ref{monster and partial}) to control the formal fibers of all minimal prime ideals, provided the initial precompletion has singleton formal fibers for certain minimal primes. 

\section*{Acknowledgements}
We would like to thank the National Science Foundation (grant DMS-1659037), Williams College, and the Clare Boothe Luce Program for providing funding for this research.

\bibliographystyle{plain}
\bibliography{references}

\end{document}